\documentclass[12pt]{amsart}
\usepackage{graphicx}
\usepackage{color}
\usepackage{latexsym}
\usepackage{amssymb}                
\usepackage{amsmath}
\usepackage{amsthm}
\usepackage{amscd}
\usepackage{enumerate} 
\usepackage{verbatim}
\usepackage{xy}
\xyoption{all}
\theoremstyle{plain}
\newtheorem{thm}{Theorem}[section]
\newtheorem{lem}{Lemma}[section]

\newtheorem{defi}{Definition}[section]
\newtheorem{prop}{Proposition}[section]

\newtheorem{corollary}{Corollary}[section]
\theoremstyle{remark}
\newtheorem{rem}{Remark}[section]

\numberwithin{equation}{section}
\def\ens{\ensuremath}                                     
\newcommand\mtb[1]{\ens{\mathbb{#1}}}                     
	   \newcommand{\Q}{\mtb{Q}}   	\newcommand{\R}{\mtb{R}}	
\newcommand{\T}{\mtb T}	   \newcommand{\Z}{\mtb{Z}}	    \newcommand{\F}{\mtb{F}}
\newcommand{\C}{\mtb{C}}    
         
\newcommand\mtc[1]{\ens{\mathcal{#1}}}           

  \newcommand{\eps}{\ens{\varepsilon}}
	
\newcommand\bld[1]{\ens{\boldsymbol{#1}}}

\newcommand{\dist}{{\bf dist}}                       

\newcommand\hl[1]{{\center \color{red} \ens{\bs\bs\bs} \\}} 
\newcommand\hsp[1]{\mbox{}\hspace{#1mm}} 
\newcommand\vsp[1]{\par \vspace{#1mm}} 
\def\bs{{\bigstar}}                     

\renewcommand{\v}{\vsp}

\newcommand{\h}{\hsp}
\newcommand{\Dst}{\displaystyle}

\newcommand{\re}{\mathrm{Re}}
\newcommand{\im}{\mathrm{Im}}

\newcommand{\fractional}[1]{|\h{-.3}|#1|\h{-.3}|}
\renewcommand{\ni}{\noindent}

\def\J{\ens{\Sigma}}
\makeatletter

\newcommand{\Rmnum}[1]{\expandafter\@slowromancap\romannumeral #1@}
\makeatother

\title[Approximate embedding into $\Z^2$]{Approximate embedding of \\ large polygons into $\Z^2$}
\author[M. Boshernitzan {}]{Michael Boshernitzan {}}
\address{Department of Mathematics, Rice University, Houston, TX~77005, USA}
\email{michael@rice.edu}
\thanks{The author was supported in part by research grant: NSF-DMS-1102298}
\date{August 3, 2012}
\begin{document}
\maketitle
\begin{abstract}
Let $\Z^2$ denote the standard lattice in the plane $\R^2$.
We prove that given a finite subset $S\subset\R^2$ and $\eps>0$, 
then for all sufficiently large dilations $t>0$  there exists 
a rotation $\rho\colon \R^2\to\R^2$ around the origin 
such that $\dist(\rho(tz),\Z^2)<\eps$, for all $z\in S$.
The result, in a larger generality, has been proved in 2006 by 
Tamar Ziegler (improving earlier results by Furstenberg, Katznelson, 
Weiss). The proof presented in the paper is short and self-contained;
it is not based on the Furstenberg correspondence principle.
\end{abstract}

\section{Conjecture and result}\label{sec:first}
For $z\in\R^k$, denote by $\fractional{z}=\dist(z,\Z^k)$ the Euclidean distance
of $z$ from the standard lattice set $\Z^k\subset\R^k$.

For a finite metric space $X=(X,d)$ and $k\geq1$, denote by $\J_k(X)$
the collection of isometric embeddings $f\colon X\to\R^k$. Set
\begin{equation}\label{eq:tau}
\tau_{k}(X)=\inf_{f\in\J_k(X)}\big(\max_{x\in X} \ \fractional{f(x)}\big).
\end{equation}

A metric space $X$ is called {\em $k$-flat}\/ if it embeds isometrically into $\R^k$ 
(i.e., if $\J(X)\neq\emptyset$).
By definition, $\tau_{k}(X)=\infty$ if $X$ is not 
$k$-flat. Clearly  $\tau_{k}(X)<\frac{\sqrt k}2$ if  $X$ is $k$-flat.

Given a metric space $X=(X,d)$ and $t>0$, we write  $tX$  for the metric space $(X,td)$.

The following theorem is restatement of a special case of T. Ziegler's result, \cite[Theorem~1.3]{Z},
which superseded  an earlier result \cite{FKW} by Furstenberg, Katznelson and Weiss.

\begin{thm}\label{conj:1}
For $k\geq2$ and any finite $k$-flat metric space $X$
\begin{equation}\label{eq:conj}
\lim_{t\to+\infty}\tau_{k}(tX)=0. 
\end{equation}
\end{thm}

We present a new proof of Theorem \ref{conj:1} which works only for even $k$.
The proof is short and self-contained, it is not based on the Furstenberg correspondence
principle (the way the approaches in the papers mentioned above are).
Our approach may lead to explicit estimates on the speed convergence  
in $\eqref{eq:conj}$ to $0$. (This direction is not pursued in this paper).

Our proof of Theorem  \ref{conj:1}  is based on the following two theorems.

\begin{thm}\label{thm:sum}
The set\, $\mtc C$ of natural numbers $k$  for which the claim of Theorem \ref{conj:1}
holds is closed under addition.
\end{thm}

\begin{thm}\label{thm:conj}
$2\in\mtc C$, i.e. for any finite $2$-flat metric space~$X$
\[
\lim_{t\to+\infty}\tau_2(tX)=0.
\]
\end{thm}

Theorem \ref{thm:conj} is an immediate consequence of 
Theorem~\ref{thm:real} in the next section. The proof of Theorem \ref{thm:sum}
is presented in Section \ref{sec:proof:sum}.

I am indepted to Professor Mihalis Kolountzakis for directing me to
Ziegler's paper \cite{Z} on the next day the first version of the
current note appeared on arXiv.

\section{Triangles and embeddings}

By a triangle we mean a metric space of cardinality 3. Since every
triangle is $2$-flat, we have the following corollary.

\begin{corollary}\label{cor:tri}
For every triangle $X$, $\Dst\lim_{t\to+\infty}\tau_2(tX)=0$.
\end{corollary}


By the separation of a finite metric space $X=(X,d)$ (notation:  $\mathbf{sep}(X)$)
we mean the minimal positive distance in it: \
$\Dst\mathbf{sep}(X)=\min_{\substack{a,b\in X\\ a\neq b}} d(a,b)$.

The following proposition shows that   the value $\tau_{2}(X)$  
does not need to be small for triangles $X$ with large separation $\mathbf{sep}(X)$ (cf.~Corollary~\ref{cor:tri}). 
\begin{prop}\label{prop:sep}
For all triangles
\[
X_t=\{(0,0),(0,t),(t+\tfrac12,0)\}\subset\R^2, \quad \text{with } t>0,
\]
we have \ $\tau_{2}(X_t)\geq\tfrac18$ \ and \ $\mathbf{sep}(X_t)=t$.
\end{prop}

In the proof of Proposition \ref{prop:sep}, we use the fact that 
the formula
\[
\fractional{x-y}=\dist(x-y,\Z^2)\quad (x,y\in\R^2)
\]
defines a pseudometric on $\R^2$.
\begin{proof}[Proof of Proposition  \ref{prop:sep}]
Let $f\in\J_2(X_t)$ and set
\[
A=(0,0), \ B=(0,t), \ C=(t+\tfrac12,0), \ D=(t,0); \quad A,B,C,D\in\R^2.
\]
We have to show that
\[
m:=\max\big(\fractional{f(A)}, \fractional{f(B)}, \fractional{f(C)}\big)\geq \tfrac18.
\]
Since $f$ is an isometry, $|C-D|_2=\frac12$ implies $|f(C)-f(D)|_2=\frac12$, and hence
\[
\tfrac12=\fractional{f(C)-f(D)}\leq\,\fractional{f(C)-f(A)}+\fractional{f(D)-f(A)}.
\]
Note that $\fractional{f(D)-f(A)}=\fractional{f(B)-f(A)}$ because the corresponding 
vectors in $\R^2$ are perpendicular and have the same length. It follows that
\begin{align*}
\tfrac12&\leq\fractional{f(C)-f(A)}+\fractional{f(B)-f(A)}\leq\\
&\leq\,2\,\fractional{f(A)}+\fractional{f(B)}+\fractional{f(C)}\leq 4m,
\end{align*}
whence $m\geq\tfrac18$, completing the proof of Proposition \ref{prop:sep}.
\end{proof}
\vsp1
\begin{rem}
One can show that $\Dst\lim_{t\to+\infty}\tau_k(X_t)=0$, for $k\geq 3$. In general,
if 
\[
X_{s,t}=\{(0,0),(0,t),(s,0)\}\subset\R^2,
\]
then  $\Dst\lim_{\substack{s\to+\infty\\t\to+\infty}}\tau_k(X_{s,t})=0$,  for $k\geq 3$.
\end{rem}
This does not extend to $k=2$ in view of Proposition \ref{prop:sep}.


\section{Central result}
Denote by $\Z^2$ the standard lattice in the plane $\R^2$.
Denote by $\mtc R$  the set of rotations  $\rho\colon\R^2\to\R^2$ around 
the origin. 
The central result of the paper is the following theorem.
\begin{thm}\label{thm:real}
Let $S\subset\R^2$ be a finite set. Then, for every $\eps>0$, there exists $T>0$  
such that for every dilation $t>T$ there exists a rotation $\rho\in \mtc R$, $\rho=\rho(\eps,t)$,  
such that
\[
\fractional{\rho(tz)}=\dist(\rho(tz),\Z^2)<\eps, \quad \text{for all }\ z\in S.
\]
\end{thm}
The above theorem is translated into Theorem \ref{thm:complex} below using the language 
of complex numbers. The proof is provided for Theorem \ref{thm:complex} 
(rather than for Theorem \ref{thm:real}) because it utilizes  the field structure 
of $\C$, and the alternative proof is somewhat longer. 

One identifies $\R^2$ with the complex plane $\C$, the lattice $\Z^2$ 
becomes the set $\Z[i]$ of Gaussian integers, and one sets the notation:
\[
\fractional z=\dist(z,\Z[i])=\min_{w\in \Z[i]}\,|z-w|, \quad \text{for }\  z\in\C,
\]
and
\[
\fractional V=\max_{1\leq k\leq r} \fractional{z_k}, \quad \text{ where }\ V=(z_k)_{k=1}^r\in\C^{r}.
\]

Denote by $\T=\{z\in \C\mid|z|=1\}$ the unit circle in $\C$.

\begin{thm}\label{thm:complex}
Let\, $V=(z_k)_{k=1}^r\in\C^{r}$ be a vector, $r\geq1$. Then, for every $\eps>0$, 
there exists $T>0$  such that for every $t>T$ there is (a rotation) $\theta\in \T$  
such that 
\[
\fractional{\theta tV}<\eps.
\]
\end{thm}
Note that \
$
\fractional{\theta tV}<\eps \iff \fractional{\theta tz_k}<\eps,  \text{ for all }1\leq k\leq r.
$

The central idea of the proof (of Theorem \ref{thm:complex}) is conveyed in
Section \ref{sec:typ} where a special case (when $V$  is typical 
in the sense of Definition \ref{def:typ} of the next section) is treated.

The proof of Theorem \ref{thm:complex} is completed in Section \ref{sec:gen}.

In the final Section 7 we state a conjecture extending Theorem 1.1 (the lattice 
$\Z^2$ is repalced by an arbitrary syndetic subset of $\R^2$).

\section{Notation and definitions}

In what follows we do not distinguish between vectors $V=(z_k)_{k=1}^r\in\C^{r}$
and $r$-tuples of its entries with appropriate $r$.
\begin{defi}\label{def:fgen}
\em   Let $\F\subset\C$ be a subfield of complex numbers. An $r$-tuple (or a vector) 
$V=(z_k)_{k=1}^r\in\C^{r}$ is called {\em$\F$-generic} if its $r$ entries are linearly
independent over $\F$.
\end{defi}
In particular, the entries of any $\F$-generic vector must be distinct.\vsp2

\ni{\bf Notation}. For any vector 
$V=(z_k)_{k=1}^r=(z_1,z_2,\ldots,z_r)\in\C^{r}$, set 
\begin{equation*}
\re(V)=(\re(z_k))_{k=1}^r\in\R^r,\qquad \im(V)=(\im(z_k))_{k=1}^r\in\R^r,
\end{equation*}
(where $\re(\cdot),\,\im(\cdot)$ stand for real and imaginary parts, respectively), and let
\[
\re\,\im(V)=(\re(z_1),\ldots,\re(z_r),\im(z_1),\ldots,\im(z_r))\in\R^{2r}
\]
denote the result of concatenation of $\re(V)$ and $\im(V)$.
\begin{defi} \label{def:typ}
\em A vector $V=(z_k)_{k=1}^r=(z_1,z_2,\ldots, z_r)\in\C^r$ is called {\em typical}
if the vector  $\re\,\im(V)\in\R^{2r}$ is $\Q$-generic in the sense of 
Definition \ref{def:fgen}.
\end{defi}
We will use the following lemma. 
\begin{lem}\label{lem:dens}
Let  $V=(z_k)_1^r\in\C^r$  be a typical $r$-tuple, and let $\eps>0$ be given. 
Then there is a positive number $L=L(V,\eps)>0$ such that for any other $r$-tuple \ 
$W=(w_k)_1^r\in\C^r$  there exists a number $s\in[0,L]$, 
such that \ $\fractional{W+sV}<\eps$  (equivalently, 
$\fractional{w_k+sz_k}<\eps$, for all $1\leq k\leq r$).
\end{lem}
\begin{proof}
Set 
\v{-9}
\begin{align*}
V'&=\re\,\im(V)\,=(v_k)_{k=1}^{2r}\, \in\R^{2r} \quad \text{and} \\
W'&=\re\,\im(W)=(u_k)_{k=1}^{2r}\in\R^{2r}
\end{align*}

In order to prove the lemma, it is enough to establish the existence of 
a positive number $L=L(V',\eps)$ (not depending on $W'$) such that
\[
\min_{0\leq s\leq L} \fractional{W'+sV'}<\eps
\]
(equivalently, $\dist(W'+sV',\Z^{2r})<\eps$).

This existence follows from the minimality of the linear flow 
\[
T^{t}X=X+tV'\h{-3}\pmod 1,\qquad X\in\R^{2r}/\Z^{2r},
\]
on the torus $\R^{2r}/\Z^{2r}$ in the direction $V'$. (Since $V$ is typical,  
all the entries $v_k$ of the vector $V'$ are linearly independent, and the minimality 
holds, see e.g. \cite[Proposition 1.5.1]{KaHa} from the dynamical point of view, 
or \cite[\S9, Example 9.3 and Excercise 9.27]{KuN}, from the number theory perspective).

Indeed, the minimality of the linear flow implies that all orbits become 
$\eps$-dense in the same finite time $L$.
\end{proof}
\section{Proof of Theorem \ref{thm:complex}, typical case}\label{sec:typ}
In this section we prove Theorem \ref{thm:complex} under the added assumption that the
$r$-tuple 
\[
V=(z_k)_{k=1}^r\in\C^r
\]
is typical (Definition \ref{def:typ}).
Let  $\eps>0$ be given. Set $W=-it\,V=(-it\,z_k)_{k=1}^r$. By Lemma~\ref{lem:dens}, 
there exists $L>0$  such that for every $t>0$ there exists $s(t)\in [0,L]$  such that
\[
\fractional{\,\big(-it+s(t)\big)V\,}<\eps/2,  \quad t\in\R.
\]
%

Since multiplication by $i$ is an isometry in $\C$, we obtain
\begin{equation}\label{eq:dist1}
\fractional{\,\big(t+is(t)\big)V\,}<\eps/2,  \quad t\in\R.
\end{equation}
For $t>0$, set 
\[
\theta_t=e^{\frac{is(t)}t}\in \T=\{z\in\C\mid |z|=1\}.
\] 
Then
\[
\theta_{t}\,t=e^{\frac{is(t)}t}t=\Big(1+\frac{is(t)}t+O(t^{-2})\Big)\cdot 
     t=t+is(t)+o(1)   \ \ \text{ (as }t\to\infty),
\]
whence, for large $t>0$, we obtain
\begin{equation}\label{eq:dist2}
\big|\,\theta_{t}tV\,-\,\big(t+is(t)\big)\,V\,\big|<\eps/2.
\end{equation}

The inequalities  \eqref{eq:dist1} and \eqref{eq:dist2} imply that
$\fractional{\,\theta_{t}tV\,}<\eps$ for large $t$,
completing the proof of Theorem \ref{thm:complex} 
(under the added assumption that $V$ is typical).

\section{The case of $V$ being $\Q[i]$-generic}\label{sec:gen}
In this section we prove Theorem \ref{thm:complex} under the added assumption that the
$r$-tuple 
\[
V=(z_k)_{k=1}^r\in\C^r
\]
is\, $\Q[i]$-generic (Definition \ref{def:fgen})  where 
$\Q[i]=\{z\in\C\mid \re(z),\im(z)\in\Q\}$. 

Note that this condition on $V$ 
($\Q[i]$-genericity) is weaker than the one imposed in the preceding section
(for $V$ to be typical).

In view of the result in the preceding section, it is enough to show that $\theta V$ is
typical for some $\theta\in\T=\{z\in\C\mid |z|=1\}$.
\begin{prop}\label{prop:gen}
Let $V=(z_k)_{k=1}^r\in\C^r$ be a $\Q[i]$-generic vector. Then for all but at most 
a countable set of $\phi\in[0,2\pi)$ the vector $e^{i\phi}V$ is typical.
\end{prop}
The proof (in the end of the section) is a combination of Lemmas \ref{lem:2typ} 
and~\ref{lem:gentyp} below.
\begin{lem}\label{lem:2typ}
Let $V=(z_k)_{k=1}^r\in\C^r$ be a $\Q[i]$-generic vector.  Then for all but at most a countable set
of $\phi\in[0,2\pi)$ the vector
\[
W_{\phi}=(e^{i\phi}z_1,e^{i\phi}z_2,\ldots,e^{i\phi}z_r,e^{-i\phi}\bar z_1,e^{-i\phi}\bar z_2,\ldots,
    e^{-i\phi}\bar z_r)\in \C^{2r}
\]
is also $\Q[i]$-generic.
\end{lem}
\begin{proof}
It is enough to show that for every non-zero vector
\[
\bld 0\neq U=(a_1,a_2,\ldots,a_r,b_1,b_2,\ldots,b_r)\in (\Q[i])^{2r}
\]
the scalar product
\[
U\cdot W_{\phi}=\sum_{k=1}^r a_ke^{i\phi}z_k+\sum_{k=1}^r b_ke^{-i\phi}\bar z_k
=e^{i\phi}\sum_{k=1}^r a_kz_k+e^{-i\phi}\sum_{k=1}^r b_k\bar z_k
\]
may vanish only for a finitely many $\phi\in[0,2\pi)$.

Indeed, the alternative is that $U\cdot W_{\phi}=0$ for all $\phi$ (since\, $U\cdot W_{\phi}$
is analytic in $\phi$), and hence both sums
\[
S_1=\sum_{k=1}^r a_kz_k; \qquad S_2=\sum_{k=1}^r b_k\bar z_k
\]
vanish  (since the functions $e^{i\phi}$ and $e^{i\phi}$ are linearly independent 
over $\C$). The $\Q[i]$-genericity of $V$ and the fact that $S_1=0$ imply that all 
$a_k=0$. On the other hand,  $S_2=0$ implies $\bar S_2=\sum_{k=1}^r \bar b_kz_k=0$,
and hence all  $b_k=0$. We conclude that $U=\bld 0$, a contradiction.
\end{proof}
\begin{lem}\label{lem:gentyp}
Let $V=(z_k)_{k=1}^r\in\C^r$ be a vector and assume that the vector
\[
W=(z_1,z_2,\ldots,z_r,\bar z_1,\bar z_2,\ldots,\bar z_r)\in\C^{2r}
\]
is $\Q[i]$-generic. Then $V$ is typical.
\end{lem}
\begin{proof}
It is easy to see that the $2r$ entries of the vector $W$ and the $2r$ entries of
the vector 
\[
U=\re\,\im(V)=(\re(z_1),\ldots,\re(z_r),\im(z_1),\ldots,\im(z_r))\in\R^{2r}
\]
generate the same vector space over $\Q[i]$. Denote by $d$ its dimension.

We have $d=2r$  because  $W$  is $\Q[i]$-generic.  It follows that $U$ is also 
$\Q[i]$-generic (otherwise we would have $d<2r$). Since $U$  is a real vector, it is 
in fact $\Q$-generic, i.e. $V$ is typical. This complete the proof of 
Lemma \ref{lem:gentyp}.
\end{proof}
\begin{proof}[Proof of Proposition {\em \ref{prop:gen}}]
The proof is obtained by combining statements of 
Lemmas \ref{lem:2typ} and~\ref{lem:gentyp}.
\end{proof}


\section{General case}\label{sec:gen}
It remains to consider the case when the vector $V$ is not $\Q[i]$-generic.
(Recall that the case of $\Q[i]$-generic $V$ is considered in the previous section).
Without loss of generality,  we may assume (after a rearrangement of the entries of $V$
if needed) that we have a representation
\[
V=(z_1,z_2,\ldots,z_m,w_1,w_2,\ldots,w_n)\in\C^{r}, \quad r=m+n, \quad m,n\geq1,
\]
where the first $m$ entries $z_1,z_2,\ldots,z_m$ are linearly independent over $\Q[i]$ 
and each of the consequent $n$ entries $w_k$ is a linear combination of the first 
$m$ ones:
\[
w_j=\sum_{k=1}^m f_{j,k}z_k;\  1\leq j\leq n; \quad  f_{j,k}\in \Q[i].
\] 

Select an integer $\Dst M>\sum_{k=1}^m\sum_{j=1}^n|f_{j,k}|$ \ such that
\[
Mf_{j,k}\in\Z[i], \text{ for all } 1\leq j\leq n, 1\leq k\leq m. 
\]

Let $\eps>0$ be given.
Since the vector $V_0=\frac1M(z_1,z_2,\ldots,z_m)\in \C^m$ is $\Q[i]$-generic, it follows 
from the result stated in the beginning of Section \ref{sec:gen} that for all 
sufficiently large $t>0$  there exists a point $\theta_t\in\T=\{z\in\C\mid|z|=1\}$ 
such that for all large $t>0$
\[
\fractional{\,\tfrac1M\theta_ttV_0\,}<\frac\eps{2mM^2},
\]
or, equivalently,
\begin{equation}\label{eq:m2}
\fractional{\,\frac{\theta_ttz_{k}}M\,}<\frac\eps{2mM^2}, \text{ for all }1\leq k\leq m.
\end{equation}
It follows that
\[
\fractional{\,\theta_ttz_{k}\,}<\frac\eps{2mM}<\eps, \text{ for all }1\leq k\leq m,
\]
and
\begin{align*}
\fractional{\,\theta_ttw_{j}\,}&=\fractional{\ \sum_{k=1}^m\, \theta_ttf_{j,k}z_k\,}
\leq\sum_{k=1}^m\,\fractional{\, \theta_tt(Mf_{j,k})\frac{z_k}M\,}\leq\\
&\leq   \sum_{k=1}^m \,(Mf_{j,k})\,\fractional{\, \frac{\theta_ttz_k}M\,}
\leq   \sum_{k=1}^m \,M^2\,\frac \eps{2mM^2}=\frac\eps2<\eps, \text{ for all }1\leq j\leq n.
\end{align*}

We conclude $\fractional{\,\theta_ttV\,}<\eps$ for large $t$, completing the proof of 
Theorem \ref{thm:complex}.

\section{Proof of Theorem \ref{thm:sum}}\label{sec:proof:sum}

Given that $m,n\in\mtc C$, we have to show that $p=m+n\in\mtc C$, i.e. that, given
a $p$-flat metric space $X=(X,d)$ with $p=m+n$, then  $\Dst\lim_{t\to+\infty}\tau_p(tX)=0$.

Without loss of generality, we assume that $X\subset\R^p=R^m\times\R^n$. 
Let $\eps>0$ be given. We have to show that there exists $\psi\in\J_p(tX)$
such that $\fractional{\psi(x)}<\eps$, for all $x\in X$.

Denote by $\pi_{1},\pi_{2}$
corresponding projections $\pi_1\colon\R^p\to\R^m$ and $\pi_2\colon\R^p\to\R^n$. Let
$X_1=\pi_{1}(X)\in\R^m$, $X_2=\pi_{2}(X)\in\R^n$.  

Since $m,n\in\mtc C$,  for large $t>0$  there are isometric embeddings
\[
\phi_{1,t}\colon tX_1\to\R^m, \qquad  \phi_{2,t}\colon tX_2\to\R^n.
\]
such that
\[
\fractional{\phi_{1,t}(u)}<\eps/2, \quad \text{for all } u\in X_1,
\]
and 
\[
\fractional{\phi_{2,t}(v)}<\eps/2, \quad \text{for all } v\in X_2.
\]
Set the maps
\[
\psi_1\colon X\to\R^m, \quad \psi_2\colon X\to\R^n,
\]
as the compositions
\[
\psi_1\colon=\phi_{1,t}\circ M_{t}\circ \pi_1, \quad \psi_2\colon=\phi_{2,t}\circ M_{t}\circ  \pi_2
\]
where  $M_t$ is a multiplication by $t$ operator.

Then the map $\psi=\psi_1\times\psi_2$ is an isometric embedding, \
$\psi\colon tX\to\R^m\times\R^n=\R^p$, \  such that 
\[
\fractional{\psi(x)}\leq\sqrt{\Big(\frac\eps2\Big)^2+\Big(\frac\eps2\Big)^2}<\eps,
\]
for all $x\in X$. (Recall that by definition $tX=(X,td)$, for metric spaces $X=(X,d)$).

This completes the proof of Theorem \ref{thm:sum}.



\begin{thebibliography}{xxx}
\bibitem{Cas} J.\ W.\ S.\ Cassels, 
{\textbf An introduction to diophantine approximation},
Cambridge University Press, 1965
\bibitem{CFS} I.\ P.\ Cornfeld, S.\ Fomin, 
Ya.\ G.\ Sinai, \textbf{Ergodic Theory}, 
Grundlehren der Mathematisches Wissenschaften \ 
[Fundamental Principles of Mathematical Sciences], 
245.  Springer-Verlag, New York, 1982.
\bibitem{Fu} H.\ Furstenberg, 
\textbf{Recurrence in Ergodic Theory and 
Combinatorial Number Theory},
Princeton University Press, 1981.
\bibitem{FKW}H.\ Furstenberg, Y.\ Katznelson and B.\ Weiss, 
Ergodic theory and configurations in sets of positive density, 
{\textit Mathematics of Ramsey Theory, Algorithms Combin.}, 
{\bf5}, Springer, Berlin (1990), pp. 184Ð198. 

\bibitem{KaHa} A.\ Katok, B.\ Hassenblatt, 
\textbf{Introduction to the Modern Theory of
Dynamical Sistems}, Encyclopedia of mathematics
and its applications, Vol 54,
Cambridge University Press, 1995.

\bibitem{KuN} L.\ Kuiper, H.\ Niederreiter,
Uniform Distribution of Sequences, 
John Wiley \& Sons, Inc., 1974.
\bibitem{Z} T.\ Ziegler, Nilfactors of 
$R^m$-actions and configurations in sets of 
positive upper density in $R^m$,
\textit{J.\ Anal.\ Math.} {\bf99} (2006), 
249--266.
\end{thebibliography}
\end{document}